\documentclass[11pt,letterpaper]{amsart}
\usepackage{amsfonts, amsmath, amssymb, amscd, amsthm, graphicx,enumerate}
\usepackage{epsfig, color}
\usepackage{hyperref}

\makeatletter
\def\blfootnote{\xdef\@thefnmark{}\@footnotetext}
\makeatother

\newtheorem{thm}{Theorem}[section]

\newtheorem{lem}[thm]{Lemma}
\newtheorem{prop}[thm]{Proposition}

\theoremstyle{definition}

\theoremstyle{remark}
\newtheorem{rem}[thm]{Remark}

\newfont{\eufm}{eufm10}

\newcommand{\e}{\varepsilon }
\renewcommand{\phi}{\varphi}

\newcommand{\Ga }{Cay(G, \mathcal A)}
\newcommand{\N}{\mathbb N}

\newcommand{\h}{\hookrightarrow _h }

\begin{document}

\title{Counting conjugacy classes in $Out(F_N)$}

\author{Michael Hull and Ilya Kapovich}

\address{\tt
Department of Mathematics,
PO Box 118105,
University of Florida,
Gainesville, FL 32611-8105
\newline
https://people.clas.ufl.edu/mbhull/ }
\email{\tt mbhull@ufl.edu}

\address{\tt Department of Mathematics, University of Illinois at
  Urbana-Champaign, 1409 West Green Street, Urbana, IL 61801, USA
  \newline http://www.math.uiuc.edu/\~{}kapovich/} \email{\tt
  kapovich@math.uiuc.edu}

\thanks{The second author was supported by the individual NSF
  grants DMS-1405146 and DMS-1710868. Both authors acknowledge the support of the conference grant DMS-1719710 ``Conference on Groups and Computation".}


\date{}
\maketitle

\begin{abstract}
We show that if a f.g. group $G$ has a non-elementary WPD action on a hyperbolic metric space $X$, then the number of $G$-conjugacy classes of $X$-loxodromic elements of $G$ coming from a ball of radius $R$ in the Cayley graph of $G$ grows exponentially in $R$. As an application we prove that for $N\ge 3$ the number of distinct $Out(F_N)$-conjugacy classes of fully irreducibles $\phi$ from an $R$-ball in the Cayley graph of $Out(F_N)$ with $\log\lambda(\phi)$ on the order of $R$ grows exponentially in $R$. 
\end{abstract}




\section{Introduction}

The study of the growth of the number of periodic orbits of a dynamical system is an important theme in dynamics and geometry. A classic and still incredibly influential result of Margulis~\cite{M} computes the precise asymptotics of the number of closed geodesics of length $\le R$ (that is, of periodic orbits of geodesic flow of length $\le R$) for a compact hyperbolic manifold.  A recent famous result of Eskin and Mirzakhani \cite{EM}, which served as a motivation for this note, shows that for the moduli space $\mathcal M_g$ of a closed oriented surface $S_g$ of genus $g\ge 2$, the number $N(R)$ of closed Teichmuller geodesics in $\mathcal M_g$ of length $\le R$ grows as $N(R)\sim \frac{e^{hR}}{hR}$ as $R\to\infty$ where $h=6g-6$.  Note that in the context of a group action, counting closed geodesics amounts to counting conjugacy classes of elements rather than counting elements displacing a basepoint by distance $\le R$. The problems about counting conjugacy classes with various metric restrictions are harder than those about counting group elements and many group-theoretic tricks (e.g. embeddings of free subgroups or of free subsemigroups) do not, a priori, give any useful information about the growth of the number of conjugacy classes.   In the context of the Eskin-Mirzakhani result mentioned above, a closed Teichmuller geodesic of length $\le R$ in $\mathcal M_g$ comes from an axis $L(\phi)$ in the Teichm\"uller space $\mathcal T_g$ of a pseudo-Anosov element $\phi\in Mod(S_g)$ with translation length $\tau(\phi)\le R$. Note that $\tau(\phi)=\log \lambda(\phi)$, where $\lambda(\phi)$ is the dilatation of $\phi$. Thus $N(R)$ is the number of $Mod(S_g)$-conjugacy classes of pseudo-Anosov elements $\phi\in Mod(S_g)$ with $\log \lambda(\phi)\le R$, where $Mod(S_g)$ is the mapping class group.

In this note we study a version of this question for $Out(F_N)$ where $N\ge 3$.  The main analog of being pseudo-Anosov in the $Out(F_N)$ context is the notion of a fully irreducible element.  Every fully irreducible element $\phi\in Out(F_N)$ has a well-defined \emph{stretch factor} $\lambda(\phi)>1$ (see \cite{BH}) which is an invariant of the $Out(F_N)$-conjugacy class of $\phi$. For specific $\phi$ one can compute $\lambda(\phi)$ via train track methods, but counting the number of distinct $\lambda(\phi)$ with $\log(\lambda(\phi))\le R$ (where $\phi\in Out(F_N)$ is fully irreducible) appears to be an unapproachable task.  Other $Out(F_N)$-conjugacy invariants such as the index, the index list and the ideal Whitehead graph of fully irreducibles \cite{CH,CM,HM1,P}, also do not appear to be well suited for counting problems. Unlike the Teichm\"uller space, the Outer space $CV_N$ does not have a nice local analytic/geometric structure similar to the Teichmuller geodesic flow. Moreover, as we explain in  Remark~\ref{rem:exp} below, it is reasonable to expect that (for $N\ge 3$) the number of $Out(F_N)$-conjugacy classes of fully irreducibles $\phi\in Out(F_N)$  with $\log\lambda(\phi)$ probably grows as a double exponential in $R$ (rather than as a single exponential in $R$, as in the case of $\mathcal M_g$). Therefore, to get an exponential growth of the number of conjugacy classes one needs to consider more restricted context, namely where the elements come from an $R$-ball in the Cayley graph of $Out(F_N)$. 

Here we obtain:

\begin{thm}\label{thm:A}
Let $N\ge 3$.  Let $S$ be a finite generating set for $Out(F_N)$.
There exist constants $\sigma>1$ and $C_2>C_1>0$, $R_0\ge 1$ such that the following hols.
For $R\ge 0$ let $B_R$ be the ball of radius $R$ in the Cayley graph of $Out(F_N)$ with respect to $S$.

Denote by $c_R$ the number of distinct $Out(F_N)$-conjugacy classes of fully irreducible elements $\phi\in B_R$ such that $C_1R\le \log\lambda(\phi)\le C_2 R$.
Then
\[
c_R\ge \sigma^R
\]
for all $R\ge R_0$.
\end{thm}

In the process of the proof of Theorem~\ref{thm:A} we establish the following general result:
\begin{thm}\label{thm:B}
Suppose $G$ has a cobounded WPD action on a hyperbolic metric space $X$ and $L$ is a non-elementary subgroup of $G$ for this action. Let $S$ be a generating set of $G$. 
For $R\ge 1$ let $b_R$ be the number of distinct $G$-conjugacy classes of elements of $g\in L$ that act loxodromically on $X$ and such that $|g|_S\le R$. Then there exist $R_0\ge 1$ and $c>1$ such that for all $R\ge R_0$
\[
b_R\ge c^R.
\]
\end{thm}
This statement is mostly relevant in the case where $G$ is finitely generated and $S$ is finite, but the conclusion of Theorem~\ref{thm:B} is not obvious even if $S$ is infinite. Theorem~\ref{thm:B} is a generalization of \cite[Theorem 1.1]{HO}, which states (under a different but equivalent hypothesis, see \cite{O}) that such $G$ has exponential conjugacy growth. The proofs of Theorem~\ref{thm:B} and \cite[Theorem 1.1]{HO} are similar and are both based on the theory of hyperbolically embedded subgroups developed in \cite{DGO}, and specifically the construction of virtually free hyperbolically embedded subgroups in [Theorem 6.14]\cite{DGO}.

Both Theorem~\ref{thm:A} and Theorem~\ref{thm:B} are derived using:
\begin{thm}\label{thm:C}
Suppose $G$ has a cobounded WPD action on a hyperbolic metric space $X$ and $L$ is a non-elementary subgroup of $G$ for this action. Then for every $n\ge 2$ there exists a subgroup $H\le L$ and a finite group $K\leq G$ such that: 
\begin{enumerate}
\item $H\cong F_n$.
\item $H\times K\h G$.
\item The orbit map $H\to X$ is a quasi-isometric embedding.
\end{enumerate}

In particular, every element of $H$ is loxodromic with respect to the action on $X$ and two elements of $H$ are conjugate in $G$ if and only if they are conjugate in $H$.
\end{thm}
Here for a subgroup $A$ of a group $G$ writing $A\h G$ means that $A$ is \emph{hyperbolically embedded in} $G$.

The proof of Theorem~\ref{thm:A} also uses, as an essential ingredient, a result of Dowdall and Taylor \cite[Theorem 4.1]{DT1} about quasigeodesics in the Outer spaces and the free factor complex.  Note that it is still not known whether the action of $Out(F_N)$ on the free factor complex is acylindrical, and, in a sense, the use of Theorem~\ref{thm:B} provides a partial workaround here.
Note that in the proof of Theorem~\ref{thm:A} instead of using Theorem~\ref{thm:C} we could have used an argument about stable subgroups having finite \emph{width}. It is proved in \cite{ADT} that convex cocompact (that is, finitely generated and with the orbit map to the free factor complex being a quasi-isometric embedding) subgroups of $Out(F_N)$ are \emph{stable} in $Out(F_N)$. In turn, it is proved in \cite{AMST} that if $H$ is a stable subgroup of a group $G$, then $H$ has finite width in $G$ (see \cite{AMST} for the relevant definitions). Having finite width is a property sufficiently close to being malnormal to allow counting arguments with conjugacy classes to go through. The proof given here, based on using Theorem~\ref{thm:C} above, is different in flavor and proceeds from rather general assumptions. Note that in the conclusion of Theorem~\ref{thm:C} the fact that $H\times K\h G$ implies that $H$ and   $H\times K$ are stable in $G$. 

Another possible approach to counting conjugacy classes involves the notion of ``statistically convex cocompact actions" recently introduced and studied by Yang, see \cite{Y1,Y2} (particularly \cite[Theorem~B]{Y2} about genericity of conjugacy classes of strongly contracting elements).  However, Yang's results only apply to actions on proper geodesic metric spaces with finite critical exponent for the action, such as, for example, the action of the mapping class group on the Teichm\"uller space. For essentially the same reasons as explained in Remark~\ref{rem:exp} below, the action of $Out(F_N)$ (where $N\ge 3$) on $CV_N$, with either the asymmetric or symmetrized Lipschitz metric, has infinite critical exponent. Still, it is possible that the statistical convex cocompacness methods may be applicable to the actions on $CV_N$ of some interesting subgroups of $Out(F_N)$.

The first author would like to thank the organizers of the conference ``Groups and Computation" at Stevens Institute of Technology, June 2017.  The second author is grateful to Spencer Dowdall, Samuel Taylor, Wenyuan Yang and Paul Schupp for helpful conversations.

\section{Conjugacy classes of loxodromics for WPD actions}
We assume throughout that all metric spaces under consideration are geodesic and all group actions on metric spaces are by isometries. Given a generating set $\mathcal{A}$ of a group $G$, we let $\Ga$ denote the Cayley graph of $G$ with respect to $\mathcal A$. In order to directly apply results from \cite{DGO} it is more convenient to consider actions on Cayley graphs. By the well known Milnor-Svarc Lemma this is equivalent to considering cobounded actions.

\begin{lem}[Milnor-Svarc]\label{lem:MS}
If $G$ acts coboundedly on a geodesic metric space $X$, then there exists $\mathcal{A}\subseteq G$ such that $\Ga$ is $G$-equivarently quasi-isometric to $X$.
\end{lem}

For a subgroup $H\leq G$ and  a subset $S\subseteq G$, we use $H\h (G, S)$ to denote that $H$ is \emph{hyperbolically embedded in $G$ with respect to $S$}, or simply $H\h G$ if $H\h(G, S)$ for some $S\subseteq G$. We refer to \cite{DGO} for the definition of a hyperbolically embedded subgroup. The only property of a hyperbolically embedded subgroup $H$ that we use is that $H$ is \emph{almost malnormal}, that is for $g\in G\setminus H$, the intersection of $H$ and $H^g$ is finite \cite[Proposition 4.33]{DGO}. Note that this implies that any two infinite order elements of $H$ are conjugate in $G$ if and only if they are conjugate in $H$.

For a metric space $X$ and an isometry $g$ of $X$, the asymptotic translation length $||g||_X$ is defined as $||g||_X=\lim_{i\to\infty}\frac{d(g^ix,x)}{i}$, where $x\in X$. It is well-known that  this limit always exists and is independent of $x\in X$. If $||g||_X>0$ then $g$ is called \emph{loxodromic}. For a group $G$ acting on $X$, a loxodromic element is called a \emph{WPD element} if for all $\e>0$ and all $x\in X$, there exists $N\in\N$ such that
\[
|\{h\in G\;|\; d(x, hx)<\e, d(g^nx, hg^nx)<\e\}|<\infty.
\]

We say the action of $G$ on $X$ is \emph{WPD} if every loxodromic element is a WPD element.

Now we fix a subset $\mathcal A\subseteq G$ such that $\Ga$ is hyperbolic and the action of $G$ on $\Ga$ is WPD. We say $g\in G$ is loxodromic if it is loxodromic with respect to the action of $G$ on $\Ga$. Each such element is contained in a unique, maximal virtually cyclic subgroup $E(g)$ \cite[Lemma 6.5]{DGO}.

\begin{lem}\cite[Corollary 3.17]{H}\label{L1}
If $g_1,..., g_n$ are non-commensurable loxodromic elements, then $\{E(g_1),..., E(g_n)\}\h (G, \mathcal A)$.
\end{lem}

A subgroup $L\leq G$ is called \emph{non-elementary} if $L$ contains two non-commensurable loxodromic elements. Let $K_G(L)$ denote the maximal finite subgroup of $G$ normalized by $L$. When $L$ is non-elementary, this subgroup is well-defined by \cite[Lemma 5.5]{H}.

The following lemma was proved in \cite{H} under the assumption that the action is acylindrical, but the proof only uses that the action is WPD.

\begin{lem}\cite[Lemma 5.6]{H}\label{L2}
Let $L$ be a non-elementary subgroup of $G$. Then there exists non-commensurable, loxodromic elements $g_1,..., g_n$ contained in $L$ such that $E(g_i)=\langle g_i\rangle\times K_G(L)$.
\end{lem}

\begin{proof}[Proof of Theorem~\ref{thm:C}]
First we note that (3) implies that every element of $H$ is loxodromic with respect to the action on $X$. Also, the fact that two elements of $H$ are conjugate in $G$ if and only if they are conjugate in $H$ follows from the fact that $H\times K$ is almost malnormal in $G$ and $K$ acts trivially on $H$ by conjugation.

We  use the construction from [Theorem 6.14]\cite{DGO}. As in  [Theorem 6.14]\cite{DGO}, we let $n=2$ since the construction from this case can be easily modified for any $n$.

By Lemma \ref{lem:MS}, we can assume $X=\Ga$ for some $\mathcal A\subseteq G$. Let $g_1,..., g_6$ be elements of $L$ given by Lemma \ref{L2}. Then each $E(g_i)=\langle g_i\rangle\times K_G(L)$ and furthermore $\{E(g_1),...,E(g_6)\}\h (G, \mathcal A)$ by Lemma \ref{L1}. Let $$\mathcal E=\bigcup_{i=1}^6 E(g_i)\setminus\{1\}.$$ Let $H=\langle x, y\rangle$ where  $x=g_1^ng_2^ng_3^n$ and $y=g_4^ng_5^ng_6^n$ for sufficiently large $n$. It is shown in \cite{DGO} that $x$ and $y$ generate a free subgroup of $G$ and this subgroup is quasi-convex in $Cay(G, \mathcal A\cup\mathcal E)$. Hence $H$ (with the natural word metric) is quasi-isometrically embedded in $Cay(G, \mathcal A\cup\mathcal E)$, and since the map $\Ga\to Cay(G, \mathcal A\cup\mathcal E)$ is 1-lipschitz $H$ is also quasi-isometrically embedded in $\Ga$. Let $K=K_G(L)$. Since $x$ and $y$ both commute with $K$, $\langle H, K\rangle \cong H\times K$. Finally, we can apply \cite[Theorem 4.42]{DGO} to get that $H\times K\h G$. Verifing the assumptions of \cite[Theorem 4.42]{DGO} is identical to the proof of \cite[Theorem 6.14]{DGO}.
\end{proof}

Note that Theorem \ref{thm:B} is an immediate consequence of Theorem \ref{thm:C}.

\section{The case of $Out(F_N)$.}

We assume familiarity of the reader with the basics related to $Out(F_N)$ and Outer space. For the background information on these topics we refer the reader to~\cite{BF,BH,CV,FM}.

In what follows we assume that $N\ge 2$ is an integer, $CV_N$ is the (volume-one normalized) Culler-Vogtmann Outer space, $\mathcal F_N$ is the free factor complex for $F_N$, $d_C$ is the asymmetric Lipschitz metric on $CV_N$ and $d^{sym}_C$ is the symmetrized Lipschitz metric on $CV_N$. When we talk about the Hausdorff distance $d_{Haus}$ in $CV_N$, we always mean the Hausdorff distance with respect to $d^{sym}_C$. For $K\ge 1$, by a $K$-quasigeodesic in $CV_N$ we mean a function $\gamma:I\to CV_N$ (where $I\subseteq \mathbb R$) is an interval, such that for all $s,t\in I, s\le t$ we have
\[
 \frac{1}{K}(t-s)-K \le d_C(\gamma(s),\gamma(t))\le K(t-s)+K.
\]
For $\epsilon>0$ we denote by $CV_{N,\epsilon}$ the $\epsilon$-thick part of $CV_N$.

We recall a portion of one of the main technical results of Dowdall and Taylor, \cite[Theorem 4.1]{DT1}:
\begin{prop}\label{prop:DT}
Let $K\ge 1$ and let $\gamma:\mathbb R\to CV_N$ be a $K$-quasigeodesic such that its projection $\pi\circ \gamma: \mathbb R\to \mathcal F_N$ is also a $K$-quasigeodesic.  There exist constants $D>0,\epsilon>0$, depending only on $K$ and $N$, such that the following holds:

If $\rho:\mathbb R\to CV_N$ is any geodesic with the same endpoints as $\gamma$ then:
\begin{enumerate}
\item We have $\gamma(\mathbb R), \rho(\mathbb R)\subset CV_{N,\epsilon}$. 
\item We have $d_{Haus}(\gamma(\mathbb R), \rho(\mathbb R))\le D$.
\end{enumerate}
 
\end{prop}
Here saying that $\gamma$ and $\rho$ have the same endpoints means that $\sup_{t\in \mathbb R} d_C^{sym}(\gamma(t),\rho(t))<\infty$.

\begin{proof}[Proof of Theorem~\ref{thm:A}]
By \cite{BF}, the action of $Out(F_N)$ on $\mathcal F_N$ satisfies the hypothesis of Theorem~\ref{thm:C}. Let $H$ be the subgroup provided by Theorem \ref{thm:C} with $L=Out(F_n)$.

We fix a free basis $\mathcal A=\{a,b\}$ for the free group $H$, and let $d_\mathcal A$ be the corresponding word metric on $H$. 

Note that the assumptions on $H$ imply that every nontrivial element of $H$ is fully irreducible.  Moreover, if we pick a base-point $p$ in $\mathcal F_N$, then there is $K\ge 1$ such that the image of every geodesic in the Cayley graph $Cay(H,\mathcal A)$ in $\mathcal F_N$ under the orbit map is a (parameterized) $K$-quasigeodesic.

 Pick a basepoint $G_0\in CV_N$.
Since the projection $\pi:(CV_N,d_C)\to \mathcal F_N$ is coarsely Lipschitz~\cite{BF}, and since the orbit map $H\to \mathcal F_N$ is a quasi-isometric embedding, it follows that the orbit map $(H,d_\mathcal A)\to (CV_N,d_C), u\mapsto u G_0$ is a $K_1$-quasi-isometric embedding for some $K_1\ge 1$. Moreover, the image of this orbit map lives in an $\epsilon_0$-thick part of $CV_N$ (where $\epsilon_0$ is the injectivity radius of $G_0$). Since on $CV_{N,\epsilon_0}$ the metrics $d_C$ and $d_C^{sym}$ are bi-Lipschitz equivalent, it follows that the orbit map $(H,d_\mathcal A)\to (CV_N,d_C^{sym}), u\mapsto u G_0$ is a $K_2$-quasi-isometric embedding for some $K_2\ge 1$.
For every $c\in \mathcal{A}^{\pm 1}$ fix a $d_C$-geodesic $\tau_c$ from $G_0$ to $cG_0$. 

Now let $\gamma: I\to Cay(H,\mathcal A)$ be a geodesic such that $\gamma^{-1}(H)=I\cap \mathbb Z$ and such that the endpoints of $I$ (if any) are integers. We then define a path $\underline{\gamma}:I\to CV_N$ as follows. Whenever $n\in \mathbb Z$ is such that $[n,n+1]\subseteq I$, then $\gamma(n)=g$ and $\gamma(n+1)=gc$ for some $c\in \mathcal A^{\pm 1}$. In this case we define $\gamma|_{[n,n+1]}$ to be $g\tau_c$.  Then for every geodesic $\gamma: I\to Cay(H,\mathcal A)$ as above the path $\underline{\gamma}:I\to CV_N$ is $K_3$-quasigeodesic, with respect to both $d_C$ and $d_C^{sym}$, for some $K_3\ge 1$ independent of $\gamma$. Moreover, $\underline{\gamma}(I)\subset CV_{N,\epsilon_1}$ for some $\epsilon_1>0$ independent of $\gamma$.

 Let $w$ be a cyclically reduced word of length $n\ge 1$ in $H$. Consider the bi-infinite $w^{-1}$-periodic geodesic $\gamma:\mathbb R\to Cay(H,\mathcal A)$ with $\gamma(0)=1$ and $\gamma(n)=w^{-1}$.   Thus the path the path $\underline{\gamma}:I\to CV_N$ is $K_3$-quasigeodesic, with respect to both $d_C$ and $d_C^{sym}$,  and $\underline{\gamma}(I)\subset CV_{N,\epsilon_1}$. Since $1\ne w\in H$, it follows that $w$ is fully irreducible as an element of $Out(F_N)$. Hence $w$ can be represented by an expanding irreducible train-track map $f:G\to G$ with the Perron-Frobenius eigenvalue $\lambda(f)$ equal to the stretch factor $\lambda(w)$ of the outer automorphism $w\in Out(F_N)$. There exists a volume-one ``eigenmetric" $d_f$ on $G$ with respect to which $f$ is a local $\lambda(f)$-homothety.  Then, if we view $(G,d_f)$ as a point of $CV_N$, then $d_C(G,Gw)=\log\lambda(w)$.  Moreover, in this case we can construct a $w$-periodic $d_C$-geodesic folding line $\rho:\mathbb R\to CV_N$ with the property that for any integer $i$, $\rho(i)=Gw^i$, hence for integers $i<j$ $d_C(Gw^i,Gw^j)=(j-i)\log\lambda(w)$. Thus, for any $i>0$ we have $d_C(G, w^{-i}G)=i\log\lambda(w)$.

The bi-infinite lines $\rho$ and $\underline{\gamma}$ are both $w^{-1}$-periodic (in the sense of the left action of $w^{-1}$) and therefore $\sup_{t\in \mathbb R} d_C^{sym}(\underline\gamma(t),\rho(t))<\infty$.

Hence, by Proposition~\ref{prop:DT}, there exist constants $D>0$ and $\epsilon>0$ (independent of $w$) such that $\rho\subset CV_{N,\epsilon}$ and that $d_{Haus}(\rho,\underline\gamma))\le D$.  The fact that $\rho\subset CV_{N,\epsilon}$ implies that $\rho$ is a $K_4$-quasigeodesic with respect to $d_C^{sym}$ for some constant $K_4\ge 1$ independent of $w$. Now for the asymptotic translation length $||w||_{CV}$, where $w$ is viewed as an isometry of $(CV_N, d_C^{sym})$, we get that one one hand (using the line $\rho$)
\[
\frac{1}{K_4}\log\lambda(w^{-1})\le ||w||_{CV} \le K_4 \log \lambda(w^{-1})
\]
and on the other hand (using the line $\underline\gamma$) that

\[
\frac{1}{K_3}n\le ||w||_{CV} \le K_3 n.
\]

Therefore

\[
\frac{1}{K_3K_4}n\le \log\lambda(w^{-1}) \le K_3K_4n.
\]
Recall also, that by a result of Handel and Mosher~\cite{HM}, there exists a constant $M=M(N)\ge 1$ such that for every fully irreducible $\phi\in Out(F_N)$ we have $\frac{1}{M}\le \frac{\log\lambda(\phi)}{\log\lambda(\phi^{-1})}\le M$.
Therefore for $w$ as above we have

\[
\frac{1}{K_3K_4M}n\le \log\lambda(w) \le K_3K_4Mn.
\]

Recall that $\mathcal A=\{a,b\}$ and that $S$ is a finite generating set for $Out(F_N)$. Put $M'=\max\{|a|_S,|b|_S\}$, so that for every freely reduced word $w$ in $H$ we have $|w|_S\le M'|w|_\mathcal A$. 

For $R>> 1$ put $n=\lfloor \frac{R}{M'}\rfloor$.  The number of distinct $H$-conjugacy classes represented by cyclically reduced words $w$ of length $n$ is $\ge 2^n$ for $n$ big enough. 
Recall two elements of $H$ are conjugate in $H$ of and only if they are conjugate in $Out(F_N)$.  Therefore we get

\[
\ge 2^n\ge 2^{ \frac{R}{M'}-1}
\] distinct $Out(F_N)$-conjugacy classes from such words $w$.  As we have seen above, each such $w$ gives us a fully irreducible element of $Out(F_N)$ with 

\[
\frac{1}{K_3K_4M}\frac{R}{2M'}\le \log\lambda(w) \le K_3K_4Mn\le K_3K_4M \frac{R}{M'},
\] 
and the statement of Theorem~\ref{thm:A} is verified.
\end{proof}

\begin{rem}\label{rem:exp}
As noted in the introduction, unlike in the mapping class group case, we expect that for $N\ge 3$ the number of $Out(F_N)$-conjugacy classes of fully irreducibles $\phi\in Out(F_N)$ with $\log\lambda(\phi)\le R$ grows as a double exponential in $R$. A double exponential upper bound follows from general Perron-Frobenius considerations. Every fully irreducible $\phi\in Out(F_N)$ can be represented by an expanding irreducible train track map $f:G\to G$, where $G$ is a finite connected graph with $b_1(G)=N$ and with all vertices in $G$ of degree $\le 3$. The number of possibilities for such $G$ is finite in terms of $N$. By \cite[Proposition~A.4]{KB}, if $f$ is as above and $\lambda:=\lambda(f)$, then 
$\max m_{ij}\le k \lambda^{k+1}$ (where $k$ is the number of edges in $G$ and $M=(m_{ij})_{ij}$ is the transition matrix of $f$). If $\log \lambda\le R$, we get $\max \log m_{ij} \le \log k + (k+1) R$ and $\max m_{ij}\le k e^{(k+1)R}$. Thus we get exponentially many (in terms of $R$) possibilities for the transition matrix $M$ of $f$. Since for a given length $L$ there are exponentially many paths of length $L$ in $G$, this yields an (a priori) double exponential upper bound for the number of train track maps representing fully irreducible elements of $Out(F_N)$ with $\log\lambda \le R$.

For the prospective double exponential lower bound we give the following explicit construction for the case of $F_3$.  Let $w\in F(b,c)$ be a nontrivial positive word containing the subwords $b^2$, $c^2$, $bc$ and $cb$.  Consider the automorphism $\phi_w$ of $F_3=F(a,b,c)$ defined as $\phi_w(a)=b$, $\phi_w(b)=c$, $\phi_w(c)=aw(b,c)$. We can also view $\phi_w$ as a graph map $f_w:R_3\to R_3$ where $R_3$ is the $3$-rose with the petals marked by $a$, $b$, $c$.  Then $f_w$ is an expanding irreducible train track map representing $\phi_w$. Moreover, a direct check shows that, under the assumptions made on $w$, the Whitehead graph of $f_w$ is connected. Additionally, for a given $n\ge 1$, ``most" positive words of length $n$ in $F(b,c)$ satisfy the above conditions and define fully irreducible automorphisms $\phi_w$. To see this, observe that the free-by-cyclic group $G_{w}=F_3\rtimes_{\phi_w} \mathbb Z$ can be rewritten as a one-relator group:
\begin{gather*}
G_w=\langle a,b,c,t| t^{-1}at=b, t^{-1}bt=c, t^{-1}ct=aw(b,c)\rangle=\\
\langle a, t | t^{-3}at^3=aw(t^{-1}at, t^{-2}at^2)\rangle.
\end{gather*}
Moreover, one can check that if $w$ was a $C'(1/20)$ word, then the above one-relator presentation of $G_w$ satisfies the $C'(1/6)$ small cancellation condition, and therefore $G_w$ is word-hyperbolic and the automorphism $\phi_w\in Out(F_3)$ is atoroidal. Since, as noted above, $\phi_w$ admits an expanding irreducible train track map on the rose with connected Whitehead graph, a result of Kapovich~\cite{Ka} implies that $\phi$ is fully irreducible. Moreover, if $|w|=L$, then it is not hard to check that  $\log \lambda(f_w)=\log\lambda(\phi_w)$ grows like $\log L$.

Since ``random" positive words $w\in F(b,c)$ are $C'(1/20)$ and contain $b^2,c^2,cb,bc$ as subwords, for sufficiently large $R\ge 1$, the above construction produces doubly exponentially many atoroidal fully irreducible automorphisms $\phi_w$ with $|w|=e^R$ and $\log\lambda(\phi_w)$ on the order of $R$.  We conjecture that in fact most of these elements are pairwise non-conjugate in $Out(F_3)$ and that this method yields doubly exponentially many fully irreducible elements $\phi$ of $Out(F_3)$ with $\log\lambda(\phi)\le R$. However, verifying this conjecture appears to require some new techniques and ideas beyond the reach of this paper. 
\end{rem}

\vspace{1cm}

\end{document}